\documentclass[11pt]{amsart}
\usepackage{amsmath, amssymb, amscd, amsthm, amsfonts, latexsym}
\usepackage{graphicx}
\usepackage{tikz-cd}
\usepackage{hyperref}
\usepackage{mathtools}
\usepackage{mathrsfs}
\usepackage{upgreek}
\usepackage{enumerate}
\usepackage{mathabx}
\usepackage{setspace}
\usepackage{color}
\usepackage{comment}
\usepackage{tikz}
\usepackage{tikz-cd}

\theoremstyle{plain}
\newtheorem{thm}{Theorem}[section]
\newtheorem{prop}[thm]{Proposition}

\newtheorem{cor}[thm]{Corollary}
\newtheorem{pro}[thm]{Properties}

\newtheorem{ques}[thm]{Question}

\theoremstyle{definition}
\newtheorem{defn}[thm]{Definition}

\newtheorem{rmk}[thm]{Remark}

\def\Z{{\mathbb Z}}

\def\cd{\protect\operatorname{cd}}
\def\cat{\protect\operatorname{cat}}
\def\cuplen{\protect\operatorname{cup}}
\def\secat{\protect\operatorname{secat}}

\def\im{\protect\operatorname{im}}
\def\rank{\protect\operatorname{rank}}

\title{On the sequential topological complexity of group homomorphisms}
\author{Nursultan Kuanyshov}

\address{{Nursultan Kuanyshov, Department of Mathematics, University of Florida, 358 Little Hall, Gainesville, FL 32611-8105, USA} and {Institute of Mathematics and Mathematical Modeling
125 Pushkin str., 050010 Almaty, Kazakhstan}}
\email{nkuanyshov@ufl.edu, kuanyshov.nursultan@gmail.com}



\begin{document}
\maketitle
\begin{abstract}
We define and develop a homotopy invariant notion for the sequential topological complexity of a map $f:X\to Y,$ denoted $TC_{r}(f)$, that interacts with $TC_{r}(X)$ and $TC_{r}(Y)$ in the same way Jamie Scott's topological complexity map $TC(f)$ interacts with $TC(X)$ and $TC(Y).$ Furthermore, we apply $TC_{r}(f)$ to studying group homomorphisms $\phi: G\to H.$ 

In addition, we prove that the sequential topological complexity of any nonzero homomorphism of a torsion group cannot be finite. Also, we give the characterisation of cohomological dimension of group homomorphisms.   
\end{abstract}

\section{Introduction}

Let $r\geq 2$ be an integer. The sequential topological complexity $TC_{r}(X)$ of a path-connected
space X was introduced by Rudyak \cite{Rud}, generalising Farber’s topological complexity \cite{Fa1}. The
motivation for these numerical invariants arises from robotics (see \cite{Fa1, Fa2, Fa3, KL, La}). They provide a measure of complexity
for the motion planning problem in the configuration space X with prescribed initial and final states, as
well as $r-2$ consecutive intermediate states. More precisely, consider the path-fibration $p_{r}: PX\to X^{r}$
that maps a path $\gamma:[0,1]\to X$ to the tuple $\big(\gamma(0),\gamma(\frac{1}{r-1}),\cdots,\gamma(\frac{r-2}{r-1}),\gamma(1)\big)$. Then $TC_{r}(X)$ is defined as the minimal integer $n$ for which $X^{r}$ can be covered by $n+1$ many open subsets $U_{0},\cdots,U_{k}$ such that $p_{r}$ admits a local section over each $U_{i}.$ If no such $n$ exists, we set $TC_{r}(X)=\infty.$ Note that $TC_{2}(X)$ recovers Farber's topological complexity.

When the mechanical system is a robotic arm, the $TC_{r}$ input on the navigation problem is not quite satisfactory. In the case of robotic arm one has a map $f:X\to Y$ of the configuration space X to the work space Y. Defining such a notion has gone through many forms. It started when Dranishnikov suggested to study the topological complexity of maps $TC(f)$ in his lectures at the workshop on Applied Algebraic Topology in Castro Urdiales in 2014. He defined it as the minimal number $k$ such that $X\times Y$ can be covered by open sets $U_{0},\cdots,U_{k}$ such that over each $U_{i}$ there is a section to the map $q:PX\to X\times Y$ defined as $q(\gamma)=(\gamma(0),f\gamma(1)).$ The drawback of his definition is that it can only apply to maps that have the path lifting property. In \cite{Pa1, Pa2} Pavesic modified above the definition to cover general maps, but his definition lacks the homotopy invariant feature. In 2020, Murillo and Wu \cite {MW} gave their notion, and J. Scott \cite {Sc} independently gave his definition and proved that his definition is equivalent to Murillo' and Wu's definition. Note that all of these definitions naturally generalizes to the sequential topological complexity of maps. In 2022, Rudyak and Soumen \cite{RS} gave their notions using relative version of LS-category of map and they conjectured their notion is equivalent to Scott's and Murillo and Wu's notions. Last year, C. Zapata and J. Gonzalez \cite{ZG} gave the generalized definitions that includes all previous notions. Further, they proved that the notions introduced by \cite{RS, Sc, MW} are equivalent. 

We extend Scott's definition for several reasons: 1) His definition makes it easy to prove basic inequalities, and we can prove that almost all known results from spaces to maps in this paper. 2) Very recently, probabilistic variants of LS-category and the sequential topological complexity of a metric space are given independently in \cite{DJ, Ju} and \cite{KW}. One can use the techniques from this paper to extend their notions from metric spaces to maps between metric spaces. These are relevant notions because they gave lower bounds for LS-category and the sequential topological complexity respectively. 

Our objective in this paper is to compute the sequential complexity of group homomorphisms. Let us recall that given a discrete group $\Gamma$, a classifying space (Eilenberg-MacLane space) $B\Gamma=K(\Gamma,1)$ is defined to be a path-connected space such that $\pi_{1}(B\Gamma)=\Gamma$ and $\pi_{k}(B\Gamma)=0$ for all $k\neq 1.$ Since $B\Gamma$ is the unique up to homotopy equivalence, we define that LS-category and the sequential topological complexity of a discrete group to be LS-category and the sequential topological complexity of its classifying space respectively, i.e. $\cat(\Gamma):=\cat(B\Gamma)$ and $TC_{r}(\Gamma):=TC_{r}(B\Gamma).$ 

Similarly, we can define the LS-category and the sequential topological complexity of group homomorphism $\phi:\Gamma\to\Lambda,$ $\cat(\phi):=\cat(B\phi)$ and $TC_{r}(\phi):=TC_{r}(B\phi)$ where $B\phi:B\Gamma\to B\Lambda$ is the classifying map between classifying spaces.   

The main results of the paper is the following theorems:

\begin{thm}[Theorem \ref{FGAG}]
Let $\phi:\Gamma\to \Lambda$ be an epimorphism of finitely generated abelian groups. 
Then $$TC_{r}(\phi)=(r-1)\cd(\phi).$$ 

In particular, if $\Gamma$ is free abelain groups, then $TC_{r}(\phi)=(r-1)(\rank(\Lambda)+k(T(\Lambda)))$ where $k(T(\lambda))$ is the Smith Normal number for given finite abelian group $T(\Lambda).$   
\end{thm}

\begin{thm}[Theorem \ref{Free}]
Let $\phi:\Gamma\to\Lambda$ be a nonzero epimorphism of free groups. Then $$TC_{r}(\phi)=\begin{cases} r-1 & \mbox{if} \,\,\, \Lambda\cong\Z \\ r & \mbox{if} \,\,\,  \Lambda\cong F_{m},m>1\end{cases}$$
\end{thm}

\begin{thm}[Proposition \ref{char. of group hom}]
Let $\phi:\Gamma\to \Lambda$ be an nonzero epimorphism of finitely generated groups $\Gamma, \Lambda$. Then $\cd(\phi)=0$ if and only if $\Gamma$ is the trivial group.    
\end{thm}

\begin{thm}[Theorem \ref{surface groups}]
 Let $\phi:\Gamma\to\Lambda$ be an epimorphism of the fundamental groups of the closed surfaces of $M_{g}$ and $M_{g'}.$ Then $$\cat(\phi)=\cd(\phi)=1$$ if and only if $\phi$ factors through the free groups via epimorphisms $\Gamma\to F_{n}\to \Lambda.$ Otherwise, $\cat(\phi)=\cd(\phi)=2.$  
\end{thm}

The paper is organized as follows. In Section 2, we recall the sectional category, Ganea-Schwarz approach, Berstein-Schwarz classes, the Lusternik-Schnirelman category, and the sequential topological complexity of space. In Section 3, we introduce the sequential topological complexity of a map borrowing the ideas from J.Scott's paper \cite{Sc}. We prove the basic inequalities in the terms of LS-category maps, cuplength, the sequential topological complexity of spaces, and composition of maps. In Section 4, we recall the cohomological dimension of group homomorphisms and the product formula. We prove Theorems 1.3 and Theorem 1.4, and give arbitrary product formula for group homomorphisms. In Section 5, we give an analogue of the D-topological complexity of space introduced by Farber and Oprea \cite{FO}. We prove the the same results hold for maps (Theorem \ref{FO}) as well. In Section 6, we prove Theorem 1.1 and Theorem 1.2. We give sharp bounds for other discrete group homomorphisms.

In the paper, we use the notation $H^*(\Gamma, A)$ for the cohomology of a group $\Gamma$ with coefficient in $\Gamma$-module $A$. The cohomology groups of a space $X$ with the fundamental group $\Gamma$ we denote as $H^*(X;A)$. Thus, $H^*(\Gamma,A)=H^*(B\Gamma;A)$ where $B\Gamma=K(\Gamma,1)$.

\section{Preliminaries}

\subsection{Sectional Category}

\begin{defn} Let $p:E\to B$ be a fibration. The sectional category of $p,$ denoted $\secat(p),$ is the smallest integer $k$ such that $B$ can be covered by $k+1$ open sets $U_{0},\cdots,U_{k}$ that each admit a partial section $s_{i}:U_{i}\to E$ of $p$.     
\end{defn}

Note that sectional category of a fibration $p: E\to B.$ was first introduced by Schwarz \cite{Sch} in 1961, under the name ``genus”. In 1978, James \cite{Ja} replaced the overworked term ``genus" by ``sectional category" and denoted the sectional category of $p$ by $\secat(p).$

We state the following the well-known properties of the sectional category. The proofs can be found in \cite{Sch, Ja}.

\begin{pro}
Let $p':E'\to B'$ be a pull-back fibration of $p:E\to B.$ a fibration. Then $\secat(p')\leq \secat(p).$      
\end{pro}

\begin{pro}
Let $p':E'\to B'$ and $p:E\to B.$ be fibrations. Then $\secat(p'\times p)\leq \secat(p')+\secat(p).$      \end{pro}

\subsection{Ganea-Schwarz's approach to $\secat$} Recall that an element of an iterated join $X_0*X_1*\cdots*X_n$ of topological spaces is a formal linear combination $t_0x_0+\cdots +t_nx_n$ of points $x_i\in X_i$ with $\sum t_i=1$, $t_i\ge 0$, in which all terms of the form $0x_i$ are dropped. Given fibrations $f_i:X_i\to Y$ for $i=0, ..., n$, the fiberwise join of spaces $X_0, ..., X_n$ is defined to be the space
\[
    X_0*_YX_1*_Y\cdots *_YX_n=\{\ t_0x_0+\cdots +t_nx_n\in X_0*\cdots *X_n\ |\ f_0(x_0)=\cdots =f_n(x_n)\ \}.
\]
The fiberwise join of fibrations $f_0, ..., f_n$ is the fibration 
\[
    f_0*_Y*\cdots *_Yf_n: X_0*_YX_1*_Y\cdots *_YX_n \longrightarrow Y
\]
defined by taking a point $t_0x_0+\cdots +t_nx_n$ to $f_i(x_i)$ for any $i$ such that $t_i \neq 0$. 

When $X_i=X$ and $f_i=f:X\to Y$ for all $i$  the fiberwise join of spaces is denoted by $*^{n+1}_YX$ and the fiberwise join of fibrations is denoted by $*_Y^{n+1}f$. 

\begin{thm}[\cite{Sch,Ja}]
Let $p:E\to B$ be a fibrarition with B normal. Then $\secat(p)\leq n$ if and only if the fibration $*_{B}^{n+1}p:*_{B}^{n+1}E\to B$ admits a section.
\end{thm}

\subsection{Berstein–Schwarz-Schwarz cohomology class}

The Berstein–Schwarz-Schwarz class of a discrete group $\Gamma$ is the first obstruction $\beta_{\Gamma}$ to a lift of $B\Gamma=K(\Gamma,1)$ to the universal covering $E\Gamma$. 
Note that $\beta_{\Gamma}\in H^1(\Gamma,I(\Gamma))$ where $I(\Gamma)$ is the augmentation ideal of the group ring $\Z\Gamma$~\cite{Be},\cite{Sch}.

\begin{thm}[Universality~\cite{DR},\cite{Sch}]\label{universal}
For any cohomology class $\alpha\in H^k(\Gamma,L)$, there is a homomorphism of $\Gamma$-modules $I(\Gamma)^k\to L$ such that the induced homomorphism for cohomology takes $(\beta_{\Gamma})^k\in H^k(\Gamma,I(\Gamma)^{\otimes k})$ to $\alpha$,  where $I(\Gamma)^{\otimes k}=I(\Gamma)\otimes\dots\otimes I(\Gamma)$ and $(\beta_{\Gamma})^k=\beta_{\Gamma}\smile\dots\smile\beta_{\Gamma}$.
\end{thm}

\subsection{Lusternik-Schnirelmann category}

\begin{defn} The Lusternik-Schnirelmann category (for short, LS-category) of a topological space $X$, denoted $\cat(X),$ is the minimal number $k$ such that $X$ is covered by $k+1$ open sets $U_{0},...,U_{k}$ that are each contractible in $X$.  
\end{defn}

\begin{defn}
The LS-category of a map $f:X\to Y$, denoted $\cat(f),$ is the minimal number $k$ such that $X$ is covered by $k+1$ open sets $U_{0},...,U_{k}$ with nullhomotopic restrictions $f|_{U_i}: U_i \to Y$ for all $i$.
\end{defn}

For a path connected space $X$, we turn an inclusion of a point $*\to X$ into a fibration $p^X_0:G_0(X)\to X$, whose fiber is known to be the loop space $\Omega X$.  The $n$-th Ganea space of $X$ is defined to be the space $G_n(X)=*_X^{n+1}G_0(X)$, while the $n$-th Ganea fibration $p^X_n:G_n(X)\to X$
is the fiberwise join $\ast^{n+1}_Xp^X_0$. Then the fiber of
$p^X_n$ is $\ast^{n+1}\Omega X$.

\subsection{The sequential topological complexity}

\begin{defn}
    Let $X$ be a path-connected space.
\begin{enumerate}
\item A {\em sequential motion planner} on a subset $Z\subset X^{r}$ is a map $s:Z\to PY$ such that $s(x_{0},x_{1},\cdots,x_{r-1})(\frac{j}{r-1})=x_{j}$ for all $j=0,\cdots,r-1.$
\item A {\em sequential motion planning algorithm} is a cover of $X^{r}$ by sets $Z_{0},\cdots,Z_{k}$ such that on each $Z_{i}$ there is some sequential motion planner $s_{i}:Z_{i}\to PY.$
\item The {\em sequential topological complexity of a space $X$}, denoted $TC_{r}(X),$ is the minimal number $k$ such that $X^{r}$ is covered by $k+1$ open sets $U_{0},\cdots,U_{k}$ on which there are sequential motion planners. If no such $k$ exists, we set $TC_{r}(X)=\infty.$
\end{enumerate}
    
\end{defn}

Note that above definition of the sequential topological complexity agrees with Rudyak's original definition \cite{Rud} using exponential corresponds of path space $X^{J_{r}},$ where $J_{r}, r\in N$ is the wedge of $r$ closed intervals $[0,1]_{i},i=1,..,r$  where the zero points $0_{i}\in [0,1]_{i}$ are identified. The sequential topological complexity of a path-connected space $X,$ where $r\geq 2,$ is defined as $TC_{r}(X)=\secat(\Updelta_{0}),$ where $$\Updelta_{0}:X^{J_{r}}\to X^{r}, \Updelta_{0}(\gamma)=(\gamma(1_{1}),\gamma(1_{2}),..,\gamma(1_{r-1}),\gamma(1_{r}))$$ is a fibration.  This correspondence is known as the exponential correspondence of homeomorphic path spaces $PX$ and $X^{J_{r}}.$
$$\begin{tikzcd}
PX \arrow[rr] \arrow[dr, "p_{r}"]& & X^{J_{r}} \arrow[ll] \arrow[dl, "\Updelta_{0}"] \\
    & X^{r} &
\end{tikzcd} $$ where $p_{r}(\gamma)=\big(\gamma(0),\gamma(\frac{1}{r-1}),\cdots,\gamma(\frac{r-2}{r-1}),\gamma(1)\big)$

 Given $\gamma(t)\in PX,$ we get the component-wise paths $\gamma_{i}(t)=\gamma \left(\dfrac{(i-1)t}{r-1} \right)\in X^{J_{r}},i=1,\cdots,r.$ On the other hand, given the component-wise paths $(\alpha_{1}(t),\cdots,\alpha_{r}(t))$ we can concatenate them the following way $\alpha_{1}(\bar{\alpha}_{2}\alpha_{3})\cdots (\bar{\alpha}_{r-1}\alpha_{r})$ to get the path in $PX$, where $\bar{\alpha_{i}}$ is the reverse path of $\alpha_{i}$ for $i=1,\cdots,r-1.$ Thus, we use these two different evaluation fibrations $p_{r}$ and $\Updelta_{0}$ for path space $PX$ interchangeably.

We lists some known properties of the sequential topological complexity. Proofs can be found in \cite{Rud, BGRT}.
\begin{pro} Let X and Y be path-conected, normal spaces. Then
\begin{enumerate}
    \item $\cat(X^{r-1})\leq TC_{r}(X) \leq \cat(X^{r})$.
    \item $TC_{r}(X\times Y)\leq TC_{r}(X)+TC_{r}(Y)$.
    \item If X is an ANR, then $\cuplen(\ker \Updelta^{*})\leq TC_{r}(X)$.
    \item $TC_{r}(X)\leq n$ if and only if $p_{r,n}$ has a section where $p_{r,n}=*_{X^{r}}^{n+1}p_{r}$ is the fiberwise join of a fibrartion $p_{r}.$ 
  
\end{enumerate}
\end{pro}

\section{The pullback of sequential topological complexity of map}

Let $f:X\to Y$ be a map. Let $X^{r}$ and $Y^{r}$ be the Cartisian product of $r$ copies of $X$ and $Y$ respectively, i.e $X^{r}:=X\times\cdots\times X$ and $Y^{r}:=Y\times\cdots\times Y.$ Let us denote $f^{r}:=f\times\cdots\times f:X^{r}\to Y^{r}$ and elements of $X^{r}$ and $Y^{r}$ are vectors $\bar{x}=(x_{0},\cdots,x_{r-1})$ and $\bar{y}=(y_{0},\cdots,y_{r-1})$ respectively.

\begin{defn} Let $f:X\to Y$ be the above map.
\begin{enumerate} 
\item A {\em sequential $f$-motion planner} on a subset $Z\subset X^{r}$ is a map $f_{Z}:\Z\to PY$ such that $f_{Z}(\bar{x})(\frac{j}{r-1})=f_{Z}(x_{0},x_{1},\cdots,x_{r-1})(\frac{j}{r-1})=f(x_{j})$ for all $j=0,\cdots,r-1.$
\item A {\em sequential $f$-motion planning algorithm} is a cover of $X^{r}$ by sets $Z_{0},\cdots,Z_{k}$ such that on each $Z_{i}$ there is some sequential $f$-motion planner $f_{i}:Z_{i}\to PY.$
\item The {\em (pullback) sequential topological complexity of map f}, denoted $TC_{r}(f),$ is the minimal number $k$ such that $X^{r}$ is covered by $k+1$ open sets $U_{0},\cdots,U_{k}$ on which there are sequential $f$-motion planners. If no such $k$ exists, we set $TC_{r}(f)=\infty.$
\end{enumerate}
\end{defn}
Note that if $r=2,$ we recover Scott's topological complexity for maps. Further, if map $f$ is identity on space $X,$ we get Rudyak's sequential topological complexity of space $X.$

We get the following analogous theorem of Scott \cite[Theorem 3.4]{Sc} for the sequential topological complexity for maps.

\begin{thm}\label{sequential TC}
 Let $f:X\to Y$ be a map and let $Z\subset X^{r}$. The following are equivalent:
 \begin{enumerate}
     \item There is a partial section $s:Z\to (f^{r})^{*}PY$ of the pullback fibrations $(f^{r})^{*}\Updelta_{0}^{Y}$.
     \item There is a sequential $f$-motion planner $f_{Z}:Z\to PY$.
     \item The projections from $f^{r}(Z)$ to the j+1 factor of $Y^{r}$ are homotopic, where $j=0,..,r-1$. 
     \item  $f^{r}|_{Z}$ can be deformed into the diagonal $\Updelta Y$ of $Y^{r}$. 
     \item There is a map $h:Z\to X$ such that $f^{r}\Updelta h$ is homotopic to $f^{r}|_{Z}$.

 \end{enumerate}
\end{thm}
\begin{proof} 
\textbf{$(1 \Rightarrow 2)$} Let $s:Z\to (f^{r})^{*}PY$ be a partial section of the pullback fibration $(f^{r})^{*}\Updelta_{0}^{Y}$ and consider the pullback diagram of the maps $f^{r}$ and $\Updelta_{0}^{Y}$
$$
\begin{tikzcd}
    (f^{r})^{*}PY \arrow[d,"(f^{r})^{*}\Updelta_{0}^{Y}"] \arrow[r, "\bar{f^{r}}"] & PY \arrow[d, "\Updelta_{0}^{Y}"]\\
    X^{r} \arrow[r, "f^{r}"] & Y^{r}
\end{tikzcd}
$$
where $(f^{r})^{*}PY=\{ (\bar{x},\gamma)\in X^{r}\times PY| \gamma(\frac{i}{r-1})=f(x_{i}) \, \text{for all} \, i=0,..,r-1.\}$

We define the sequential $f$-motion planner $f_{Z}:Z\to PY$ by $f_{Z}=\bar{f^{r}} s$. 
It is easy to see that 
$$\Updelta_{0}^{Y}f_{Z}=\Updelta_{0}^{Y}(\bar{f^{r}})s=(f^{r})((f^{r})^{*}\Updelta_{0}^{Y})s=(f^{r})Id_{Z}=f^{r}|_{Z}.$$ Hence, $f_{Z}$ is a sequential $f$-motion planner on Z.

\textbf{$(2 \Rightarrow 1)$} It just follows from the definition of pullback of maps. (See \cite[Theorem 3.4]{Sc}).

\textbf{$(2 \Rightarrow 3)$} Let $pr_{j}:f^{r}(Z)\subset Y^{r}\to Y$ be a projection onto the $j^{th}$ factor of $Y^{r}$. It suffices to show that $pr_{j}$ is homotopic to $pr_{j+1}$ for all $j=0,..,r-2.$ By (2), there is a sequential $f$-motion planner on $Z$ and let $\bar{x}=(x_{0},\cdots,x_{r-1})\in Z.$

We define the homotopy $$H_{j}(\bar{x},t):=f_{Z}(\bar{x})\big(\frac{t+j}{r-1}\big).$$ Then 
$H_{j}(\bar{x},0)=f_{Z}(\bar{x})(\frac{j}{r-1})=f(x_{j})=pr_{j}(f^{r}(\bar{x}))$
and
$H_{j}(\bar{x},1)=f_{Z}(\bar{x})(\frac{j+1}{r-1})=f(x_{j+1})=pr_{j+1}(f^{r}(\bar{x})).$ Thus, all projections from $f^{r}(Z)$ to the $j^{th}$ factor of $Y^{r}$ are homotopic for $j=0,\cdots,r-1.$

\textbf{$(3 \Rightarrow 4, 5)$} Since any two projections from $f^{r}(Z)$ to the $j^{th}$ factor of $Y^{r}$ are homotopic, we fix a homotopy $H_{j}:f^{r}(Z)\times I\to Y^{r}$ from $pr_{j}$ to $pr_{j+1}$ for $j=0,\cdots,r-2.$ For given $\bar{x}\in Z,$ this homotopy $H_{i}(f^{r}(\bar{x}),t)$ is a path from $f(x_{j})$ to $f(x_{j+1})$ and denote this path as $\alpha_{\bar{x}}^{j}.$ We define a concatenation of paths as $\gamma_{\bar{x}}^{i}=*^{i}_{j=0}\alpha_{\bar{x}}^{j}.$ Note that $\gamma_{\bar{x}}^{i}$ is a path from $f(x_{0})$ to $f(x_{i}).$ Now we define a homotopy $H:f^{r}(Z)\times I\to Y^{r}$ as

$$H(f^{r}(\bar{x}),t)=(H_{0}(f^{r}(\bar{x}),t(1-t)),\gamma_{\bar{x}}^{1}(1-t),\cdots,\gamma_{\bar{x}}^{r-1}(1-t)).$$

Then $H(f^{r}(\bar{x}),0)=(H_{0}(f^{r}(\bar{x}),0),\gamma_{\bar{x}}^{1}(1),\cdots,\gamma_{\bar{x}}^{r-1}(1))=(f(x_{0}),f(x_{1}),\cdots,f(x_{r-1}))=f^{r}(\bar{x})$ and 
$H(f^{r}(\bar{x}),1)=(H_{0}(f^{r}(\bar{x}),0),\gamma_{\bar{x}}^{1}(0),\cdots,\gamma_{\bar{x}}^{r-1}(0))=(f(x_{0}),f(x_{0}),\cdots,f(x_{r-1}))\in \Updelta(Y).$ This gives a deformation of $f^{r}|_{Z}$ into $\Updelta(Y),$ showing (4). Also observe that $h$ can be chosen to be $pr_{0}:Z\subset X^{r}\to X$ is the projection map into the first coordinate.  

\textbf{$(5 \Rightarrow 4)$} Note that the image of $f^{r}\Updelta h$ will always be contained in $\Updelta(Y).$

\textbf{$(4 \Rightarrow 2)$} Let $H:Z\times I\to Y^{r}$ be a deformation of $f^{r}|_{Z}$ to $\Updelta(Y).$ We define a map $f_{Z}:Z\to Y^{r}$ 
$$f_{Z}(\bar{x})(t)=\begin{cases} p_{j}H(\bar{x}, 2(r-1)t-2j) & \mbox{if} \, \, \frac{j}{r-1}\leq t \leq \frac{2j+1}{2(r-1)} \\ p_{j+1}H(\bar{x}, 2j+2-2(r-1)t) & \mbox{if} \, \, \frac{2j+1}{2(r-1)}\leq t \leq \frac{j+1}{(r-1)}\end{cases}$$
where $j=0,\cdots,r-2$ and $pr_{j}$ and $pr_{j+1}$ are projection of $Y^{r}$ to $Y$ with $j^{th}$ and $(j+1)^{th}$ coordinates respectively. Then $f_{Z}$ is well-defined and continous since $H(\bar{x},t)\in \Updelta(Y)$ for all $\bar{x}\in Z.$ Moreover, 
$$f_{Z}(\bar{x})(\frac{j}{r-1})=pr_{j}H(\bar{x}, 2(r-1)\frac{j}{r-1}-2j)=pr_{j}H(\bar{x},0)=pr_{j}(f^{r}(\bar{x}))=f(x_{j}),$$ and 
$$f_{Z}(\bar{x})(\frac{j+1}{r-1})=pr_{j+1}H(\bar{x}, 2j+2-2(r-1)\frac{j+1}{r-1})=pr_{j+1}H(\bar{x},0)=pr_{j+1}(f^{r}(\bar{x}))=f(x_{j+1})$$
so that $f_{Z}$ is an sequential $f$-motion planner on Z.

\end{proof}

We can also define the sequential topological complexity of a map $f:X\to Y,$ $TC_{r}(f):=\secat((f^{r})^{*}\Updelta_{0}^{Y}).$ It was well-defined by Theorem \ref{sequential TC} part 2. 
Since sectional category is homotopy invariant \cite{Sch}, the sequential topological complexity of map is also homotopy invariant.  

We will use the following proposition on the proof of main results in this chapter. This result is a sequential analogue of \cite[Proposition 3.8]{Sc} which relates the LS-category of map and the sequential topological complexity of map.

\begin{prop}\label{TC}
    Let $f:X\to Y$ be a map. Then

\noindent 1. $TC_{r}(f)\leq \min \{TC_{r}(X), TC_{r}(Y)\}$ \\
2. $\cat(f^{r-1})\leq TC_{r}(f)\leq \cat(f^{r})$
    
\end{prop}

\begin{proof}
1.  It is clear that $TC_{r}(f)\leq TC_{r}(Y)$ by the sectional category. 
 Now suppose that $TC_{r}(X)=k$ and let $U_{0},U_{1},...,U_{k}$ be the open cover of $X^{r}$ with motion planners $s_{i}:U_{i}\to X^{r}.$ We define $f_{i}:U_{i}\to PY$ by $f_{i}=fs_{i}.$ Then it is easy to see that $$f_{i}(x_{0},x_{1},..,x_{r-1})(\frac{j}{r-1})=f(s_{i}(x_{0},...,x_{r-1})(\frac{j}{r-1}))=f(x_{j})$$ for all $j=0,..,r-1.$ 
Therefore, $f_{i}$ are the sequential $f$-motion planners and $TC_{r}(f)\leq k.$

2. First we prove that $\cat(f^{r-1})\leq TC_{r}(f).$ Suppose $TC_{r}(f)=k$ and let $U_{0},U_{1},..,U_{k}$ be an open cover of $X^{r}$ with sequential $f$-motion planners $f_{i}:U_{i}\to PY.$ Let $a$ and $b$ be a choice of base points for $X$ and $Y$ respectively such that $f(a)=b.$ Let $\gamma_{j}$ be paths from $b$ to $f(x_{j})$ for all $j=1,..,r-1$ as we constructed in the proof of Theorem \ref{sequential TC} \textbf{$(3 \Rightarrow 4, 5)$}. We define $V_{i}=\{\bar{x}=(x_{1},..,x_{r-1})|(a,x_{1},..,x_{r-1})\in U_{i} \}.$ Then $V_{0},..,V_{k}$ is clearly an open cover for $X^{r-1}.$ Now we define a nullhomotopy $H_{i}:V_{i}\times I\to Y^{r-1}$ of $f^{r-1}$ on $V_{i}$ by $$H_{i}(\bar{x},t)=(f_{i}(a,\bar{x})(t(1-t)), \gamma^{1}(1-t),...,\gamma^{r-2}(1-t)).$$

Then $$H_{i}(\bar{x},0)=(f(a),\gamma^{1}(1),..,\gamma^{r-1}(1))=(b,f(x_{1}),...,f(x_{r-1}))$$
and 
$$H_{i}(\bar{x},1)=(f(a),\gamma^{1}(0),...,\gamma^{r-1}(0))=(b,b,...,b)$$

Therefore, $\cat(f^{r-1})\leq k.$

Next, we show that $TC_{r}(f)\leq \cat(f^{r}).$ Suppose that $\cat(f^{r})=k$ and let $U_{0},U_{1},..,U_{k}$ be an open cover of $X^{r}$ such that each $f^{r}|_{U_{i}}$ is nullhomotopic. Let $H_{i}:U_{i}\times I\to Y^{r}$ be a nullhomotopy of $f^{r}|_{U_{i}}$ to some $\bar{y}^{i}=(y_{0}^{i},..,y_{r-1}^{i})\in Y^{r}.$ Since $Y$ is a path-connected, let $\beta_{j}^{i}$ be paths from $y_{j}^{i}$ to $y_{j+1}^{i}$ for all $j=0,..,r-2.$  Now we define $f_{i}:U_{i}\to PY$ by 

$$f_{i}(\bar{x})(t)=\begin{cases} pr_{j}(H_{i}(\bar{x}, 3(r-1)t-3j)), & \mbox{if} \, \, \frac{j}{r-1}\leq t \leq \frac{3j+1}{3(r-1)} 
\\ \beta_{j}^{i}(3(r-1)t-3j-1) & \mbox{if} \, \,  \frac{j+1}{3(r-1)}\leq t \leq \frac{3j+2}{3(r-1)} 
\\ pr_{j+1}(H_{i}(\bar{x}, 3(r-1)t-3j-2)) & \mbox{if} \, \,  \frac{3j+2}{3(r-1)}\leq t \leq \frac{j+1}{r-1} 
\\pr_{r-2}(H_{i}(\bar{x}, 3(r-1)t-3(r-2))) & \mbox{if} \, \,  \frac{r-2}{r-1}\leq t \leq \frac{3(r-2)+1}{3(r-1)} 
\\ \beta_{r-2}^{i}(3(r-1)t-3(r-2)-1) & \mbox{if} \, \,  \frac{3(r-2)+1}{3(r-1)}\leq t \leq \frac{3(r-2)+2}{3(r-1)} 
\\ pr_{r-1}(H_{i}(\bar{x}, 3(r-1)-3(r-1)t)) & \mbox{if} \, \, \frac{3(r-2)+2}{3(r-1)}\leq t \leq 1
\end{cases}$$

where $j=0,..,r-3$, and $pr_{j}:Y^{r}\to Y,$ $pr_{r-2}:Y^{r}\to Y,$ and $p_{r-1}:Y^{r}\to Y,$ are projections into $j+1$, $r-1$ and $r$ coordinates, respectively. 

Since $H_{i}(\bar{x}, 1)=\bar{y}^{i}=(y_{0}^{i},..,y_{r-1}^{i}),$ $f_{i}$ is continuous. Furthermore, 

$f_{i}(\bar{x})(\frac{j}{r-1})=pr_{j}(H_{i}(\bar{x}, 3(r-1)\frac{j}{r-1}-3j))=pr_{j}(f^{r})(\bar{x})=f(x_{j})$ for all $j=0,...,r-3$,

$f_{i}(\bar{x})(\frac{r-2}{r-1})=pr_{r-2}(H_{i}(\bar{x}, 3(r-1)\frac{r-2}{r-1}-3(r-2)))=p_{r-2}(f^{r})(\bar{x})=f(x_{r-2})$, and 

$f_{i}(\bar{x})(1)=p_{r-1}(H_{i}(\bar{x}, 3(r-1)-3(r-1)1))=p_{r-1}(f^{r})(\bar{x})=f(x_{r-1})$

so that the $f_{i}$ are sequential $f$-motion planners. Therefore, $TC_{r}(f)\leq k.$
\end{proof}

\begin{rmk}
This proof recovers Jamie Scott's proof (\cite{Sc}, Proposition 3.8) when $r=2.$    
\end{rmk}

\begin{thm}\label{H space}
    Let $f:G\to H$ be any map such that $G$ is an topological group. Then $TC_{r}(f)=\cat(f^{r-1}).$
\end{thm}
\begin{proof}
 By Proposition \ref{TC}, it suffices to show that $TC_{r}(f)\leq\cat(f^{r-1}).$    
Let $e\in G$ be the unit. Let us denote $\bar{e}=(e,\cdots,e) \in G^{r-1}$ and $(e,\bar{e})\in G\times G^{r-1}.$  
Suppose that $\cat(f^{r-1})=k$ and let $U_{0},\dots,U_{k}$ be open covers of $G^{r-1}$ together nullhomotopies $H_{i}:U_{i}\times I\to H^{r-1}$ of $f^{r-1}|_{U_{i}}$ to $f^{r-1}(\bar{e}).$ Indeed, this gives a path from $f^{r-1}(\bar{b})$ to $(f(a),\cdots,f(a))=f(a)f^{r-1}(\bar{e}),$  which is continous and depends on only $\bar{b}\in U_{i}$ since  $\gamma_{\bar{b}}(t)=f(a)H_{i}(a^{-1}\bar{b},t)$

Now we define the open cover of $G^{r}=G\times G^{r-1}$ as follows

$W_{i}=\{(a,\bar{g})\in (G\times G^{r-1})| a^{-1}\bar{b}\in U_{i} \}$ for all $i=0,\cdots,k.$ Clearly, $\{W_{i}\}$ are open subsets and covers the topological group $G^{r}.$ 

Note that path between two vectors can be thought as collection of component-wise paths, i.e, $\gamma_{\bar{b}}(t)=(\gamma_{b_{1}}(t),\cdots,\gamma_{b_{r-1}}(t)).$
Abusing notation, we let $\gamma_{\bar{b}}(t)=(1_{f(a)},\gamma_{b_{1}}(t),\cdots,\gamma_{b_{r-1}}(t))$ where $1_{f(a)}$ is constant path at $f(a)$ for all $t\in [0,1].$ 

Thus, we obtain that the map $f^{r}$ restricts on each open set $W_{i}$,  $f^{r}|_{W_{i}}$ deforms into the diagonal $\Updelta(Y)$ of $Y^{r}$,  namely $(f(a),f(a),..,f(a)).$ By Theorem \ref{sequential TC}, we get an sequential $f$-motion planner on $W_{i}$ for each $i=0,\cdots,k$ (See, Theorem \ref{sequential TC}, proof of $(4 \Rightarrow 2)$).

Therefore, $TC_{r}(f)\leq k.$

\end{proof}

\begin{rmk}
    Note that the statement is true for $H$-space. Because of notational difficulties, we prove here only for topological groups, which suffices us to prove our main results. 
\end{rmk}

The cuplength lower bound for topological complexity of a map was proven in \cite{Sc} for map with an ANR X with coefficients in any $\pi_{1}(Y)\times\pi_{1}(Y)$ module. The exact same proof works for the sequential topological complexity of a map. To completeness of our work, we state it here and give the proof:

\begin{thm}\label{cup length}
Let $f:X\to Y$ be any map with X an ANR. Suppose that $u_{i}\in \tilde{H}^{*}(X^{r};A_{i})$ are such that $u_{i}\in \ker(\Updelta_{X}^{*})\cap \im(f^{r})^{*}$ and $u_{0}\cup\cdots\cup u_{k}\neq 0,$ where each $A_{i}$ are $\pi_{1}(Y)^{r}$ modules. Then $TC_{r}(f)\geq k+1.$ Indeed,
$$\cuplen(\ker(\Updelta_{X}^{*})\cap \im(f^{r})^{*})\leq TC_{r}(f).$$
\end{thm}

\begin{proof}
  Suppose $TC_{r}(f)\leq k$ so that there are open sets $U_{0},\cdots,U_{k}$ covering $X^{r}$ such that on each $U_{i}$ there is a map $h:U_{i}\to X$ satisfying $(f^{r})\Updelta h\cong (f^{r})|_{U_{i}}$ (see Thoerem \ref{sequential TC} part (5)). Thus, for each $i$ we have a commutative diagram 
$$
  \begin{tikzcd}
      \cdots \arrow[r] & H^{n}(X^{r},U_{i};A_{i}) \arrow[r,"q_{i}^{*}"] & H^{n}(X^{r};A_{i}) \arrow[r,"j_{i}^{*}"] & H^{n}(U_{i};A_{i}) \arrow[r] & \cdots \\
& & & H^{n}(X;A_{i}) \arrow[u, "h_{i}^{*}"] & \\
& & H^{n}(Y^{r};A_{i}) \arrow[uu,"(f^{r})^{*}"] \arrow[r, "(f^{r})^{*}"] & H^{n}(X^{r};A_{i}) \arrow[u, "\Updelta_{X}^{*}"] & \\       \end{tikzcd}
$$

where the top row is the long exact sequence of the pair $(X^{r},U_{i}),$ and the maps $q_{i}$ and $j_{i}$ are inclusion maps.

Now let $u_{0},..,u_{k}\in \ker(\Updelta_{X}^{*})\cap \im(f^{r})^{*}$ where each $u_{i}$ is taken in coefficients $A_{i}$ and has degree at least 1. Then there is some $v_{i}\in H^{n}(Y^{r};A_{i})$ such that $(f^{r})^{*}v_{i}=u_{i}.$  Since each $u_{i}\in \ker\Updelta_{X}^{*},$ it follows that 

$$j_{i}^{*}u_{i}=j_{i}^{*}(f^{r})^{*}v_{i}=h_{i}^{*}\Updelta_{X}^{*}(f^{r})^{*}v_{i}=h_{i}^{*}\Updelta_{X}^{*}u_{i}=h_{i}^{*}(0)=0$$

so that $u_{i}\in \ker j_{i}^{*}.$ By exactness, there is some $w_{i}\in H^{*}(X^{r},U_{i};A_{i})$ such that $q_{i}^{*}w_{i}=u_{i}.$ Now consider $w_{0}\cup\cdots\cup w_{k},$ which lies in $$H^{*}(X^{r},U_{0}\cup\cdots\cup U_{k};A_{0}\otimes\cdots\otimes A_{k})=H^{*}(X^{r},X^{r};A_{0})\otimes\cdots\otimes A_{k})=0$$

so that $w_{0}\cup\cdots\cup w_{k}=0;$ hence, $$u_{0}\cup\cdots\cup u_{k}=q_{0}^{*}w_{0}\cup\cdots\cup q_{k}^{*}w_{k}=q^{*}(w_{0}\cup\cdots\cup w_{k})=0$$ 
where $q:X^{r}\to (X^{r},X^{r})$ is the inclusion map of pairs. Thus, the theorem follows.
\end{proof}

Similarly, the proof of topological complexity of composite maps can be identically applied for the sequential topological complexity of composite maps.
\begin{prop}\label{Composition}
   $TC_{r}(gf)\leq \min \{TC_{r}(g), TC_{r}(f) \}$ for any maps $f:X\to Y$ and $g:Y\to Z.$ 
\end{prop}

\begin{proof}
Suppose that $TC_{r}(f)=k$ so that there is an open cover $U_{0},\cdots,U_{k}$ of $X^{r}$ such that each restriction $(f^{r})|_{U_{i}}$ can be deformed into $\Updelta(Y)$ via some homotopy $F_{i}:U_{i}\times I\to Y^{r}.$ Now define a homotopy $H_{i}:U_{i}\times I\to Z^{r}$ by $H_{i}=(g^{r})F_{i}$, which defines a deformation of $(gf)^{r}|_{U_{i}}$ into $\Updelta(Z);$ hence, $TC_{r}(gf)\leq k=TC_{r}(f).$ 

Now suppose that $TC_{r}(g)=k$ so that there is an open cover $V_{0},\cdots,V_{k}$ of $Y^{r}$ such that each restriction $(g^{r})|_{V_{i}}$ can be deformed into $\Updelta(Z)$ via some homotopy $G_{i}:V_{i}\times I\to Z^{r.}$  Now let $U_{i}=(f^{r})^{-1}(V_{i})$ and define a homotopy $H_{i}:U_{i}\times I\to Z^{r}$ by $H_{i}(\bar{x},t)=G_{i}(f^{r}(\bar{x}),t),$ which defines a deformation of $(gf)^{r}|_{U_{i}}$ into $\Updelta(Z).$ Therefore, $TC_{r}(gf)\leq k=TC_{r}(g).$ 

\end{proof}

\begin{cor}\label{retraction}
 Let $f:X\to Y$ be a map, and let $r:X\to A$ be a retraction. 
Then $TC_{r}(f)=TC_{r}(\bar{f})$ where $\bar{f}:=f|_{A}$.
\end{cor}
\begin{proof}
 By Proposition \ref{Composition}, we obtain that $TC_{r}(f)\leq TC_{r}(\bar{f}).$   

 Since $\bar{f}=f|_{A}$, we have $TC_{r}(\bar{f})\leq TC_{r}(f).$
\end{proof}

\section{Cohomological dimension of group homomorphisms}

Let $\Gamma$ be a discrete group. A projective resolution $P_{*}(\Gamma)$ of $\Z$ for the group $\Gamma$ is an exact sequence of projective $\Gamma-$ modules
$$\cdots\to P_{n}(\Gamma)\to P_{n-1}(\Gamma)\to\cdots P_{2}(\Gamma)\to P_{1}\to \Z\Gamma\to \Z.$$

The cohomology of $\Gamma$ with coefficients in a $\Gamma-$ module M can be defined as homology of the cochain complexe $Hom_{\Gamma}(P_{*}(\Gamma),M)$.

We recall the {\em cohomological dimension} $\cd(\phi)$ of a group homomorphism $\phi:\Gamma\to\Lambda$ was introduced by Mark Grant ~\cite{Gr} as the maximum of $k$ such that
there is a $\Z\Lambda$-module $M$ with  the nonzero  induced  homomorphism $\phi^*:H^k(\Lambda,M)\to H^k(\Gamma,M)$. When $\phi$ is the identity homomorphism, we recover the classical cohomological dimension of a discrete group $\Gamma$, the shortest length of projective resolution $P_{*}(\Gamma)$ \cite{Br}. The following Theorem of Dranishnikov and De Saha gives the chacterisation of group homomorphism in term of the projective resolutions. 

\begin{thm}[\cite{DD}] Let $\phi:\Gamma\to\Lambda$ be a group homomorphism with $\cd(\phi)=n$ and let $\phi_{*}:(P_{*}(\Gamma),\partial_{*})\to (P_{*}(\Gamma),\partial'_{*})$ be the chain map between projective resolutions of $\Z$ for $\Gamma$ and $\Lambda$ induced by $\phi$. Then $\phi_{*}$ is chain homotopic to $\psi_{*}:P_{*}(\Gamma)\to P_{*}(\Lambda)$ with $\psi_{k}=0$ for $k>n.$
\end{thm}

By Theorem 3.1 in \cite{DK}, we assume that group homomorphism $\phi:\Gamma\to \Lambda$ is an epimorphism. We get the following characterisation of group epimorphism.  

\begin{prop}\label{char. of group hom}
Let $\phi:\Gamma\to \Lambda$ be an nonzero epimorphism of finitely generated groups $\Gamma, \Lambda$. Then $\cd(\phi)=0$ if and only if $\Gamma$ is the trivial group.    
\end{prop}
\begin{proof}
Suppose $\Gamma$ is the trivial group. Then $\phi^{*}:\tilde{H^{*}}(B\Lambda;M)\to \tilde{H^{*}}(B\Gamma;M)$ is the trivial homomorphism for all $\Lambda-$ modules M. Thus, we obtain that $\cd(\phi)=0.$

We prove this direction in two steps:

{\em Step 1.} Suppose $\Gamma$ is a finite group. Since $\phi(\Gamma)\neq 0,$ by Theorem \cite{Ku1} we obtain that $\cd(\phi)=\infty.$ Thus, it is a contradiction to have $\cd(\phi)=0.$ Hence, $\Gamma$ is the trivial group.

{\em Step 2.} Suppose $\Gamma$ is a torsion-free group. We obtain the following commutative diagram 
$$
\begin{tikzcd}
\cdots \arrow[r,"\partial_{4}"] & P_{3}(\Gamma) \arrow[r,"\partial_{3}"] \arrow[d,"\phi_{3}"] \arrow[r,"\partial_{3}"] & P_{2}(\Gamma) \arrow[r,"\partial_{2}"] \arrow[d,"\phi_{2}"] & P_{1}(\Gamma) \arrow[d,"\phi_{1}"] \arrow[r,"\partial_{1}"] & \Z\Gamma \arrow[r,"\epsilon"]\arrow[d, "\phi_{0}"] & \Z \arrow[d,"1"] \arrow[r] & 0\\
\cdots \arrow[r,"\partial'_{4}"] & P_{3}(\Lambda) \arrow[r,"\partial'_{3}"] & P_{2}(\Lambda)\arrow[r,"\partial'_{2}"] & P_{1}(\Lambda) \arrow[r,"\partial'_{1}"] & \Z\Lambda \arrow[r,"\epsilon'"] & \Z \arrow[r] & 0 
\end{tikzcd}
$$
where the augmentation homomorphisms $\epsilon(g)=1$ and $\epsilon'(g')$ for all $g\in\Gamma$ and all $g'\in\Lambda.$

Note that $P_{*}(\Gamma)$, $\Z\Gamma,$ and $I(\Gamma)$ are $\Gamma-$ modules, but also $\Lambda-$ modules via $\phi.$

The Berstein–Schwarz class of $\Lambda$, $\beta_{\Lambda}\in H^{1}(B\Lambda;I(\Lambda))$,  is given by the cocycle $f_{\Lambda}:P_{1}(\Lambda)\to I(\Lambda),f_{\Lambda}(c')=\partial'_{1}(c')$ where $I(\Lambda)$ is the augmentation ideal of $\Z\Lambda$.

Since $\phi_{*}$ is nonzero surjective, for each $c'\in P_{1}(\Lambda)$, there exists $c\in P_{1}(\Gamma)$ with $c'=\phi(c).$ Since the above diagram is commutative, we get the nonzero cocycle $f_{\Lambda}\phi_{1}\in Hom_{\Lambda}(P_{1}(\Gamma),I(\Lambda))$. Hence, this implies $\phi^{*}:H^{1}(B\Lambda; I(\Lambda))\to H^{1}(B\Gamma;I(\Lambda))$
is not trivial, i.e $\phi^{*}(\beta_{\Lambda})\neq 0.$ In other words, the cohomological class $\phi^{*}(\beta{\Lambda})$ is given by a nonzero cocycle $f_{\Lambda}\phi_{1}.$ This contradicts $\cd(\phi)=0.$ Hence, $\Gamma$ is the trivial group. 
\end{proof}

\begin{rmk}
 {\em Step 2} covers {\em Step 1} since the proof uses only nonzero epimorphism. However, for finite groups one does not need to introduce the Berstein-Schwarz class.  
\end{rmk}

\begin{cor}\label{free group}
 Let $\phi:F_{n}\to F_{m}$ be an epimorphism between free groups with $n$ and $m$ generators respectively. Then $\cat(\phi)=\cd(\phi)=1$. Moreover, the Berstein–Schwarz classes of $F_{n}$ and $F_{m}$ are related.    
\end{cor}

\begin{proof}
 By Proposition \ref{char. of group hom}, it is clear that $\cd(\phi)\geq 1$ since $F_{n}$ is not trivial.
Since $\cd(\phi)\leq \cat(\phi)\leq \cat(F_{n})=1,$ we obtain that $\cat(\phi)=\cd(\phi)=1.$

We recall that the Berstein–Schwarz class of discrete group $\Gamma$ is also defined as the image of the connecting homomorphism $\delta:H^{0}(B\Gamma;\Z)\to H^{1}(B\Gamma;I(\Gamma))$, i.e., $\delta(1)=\beta_{\Gamma},$ where $I(\Gamma)$ is the augmentation ideal of the group ring $\Z\Gamma$ in \cite{DR}.

Since $\phi:F_{n}\to F_{m}$ is an epimorphism, there exists a section $s:F_{m}\to F_{n}$ with $\phi s=1_{F_{m}}.$ 

We have the following commutative diagram of the short exact sequences 

$$\begin{tikzcd}
 0 \arrow[r] & I(F_{n}) \arrow[d, "\bar{\phi}"] \arrow[r] & \Z F_{n} \arrow[d, "\bar{\phi}"] \arrow[r, "\epsilon"] & \Z \arrow[d, "1"] \arrow[r] & 0 \\
  0 \arrow[r] & I(F_{m}) \arrow[r] \arrow[bend right=60,swap]{u}{\bar{s}} & \Z F_{m} \arrow[r, "\epsilon"] \arrow[bend right=60,swap]{u}{\bar{s}} & \Z \arrow[r] \arrow[u]& 0 
\end{tikzcd}
$$
where $\bar{\phi}$ and $\bar{s}$ are induced maps of $\phi$ and $s$ respectively.

The above short exact sequences give the following commutative diagrams

$$\begin{tikzcd}
H^{0}(BF_{n};\Z) \arrow[dd,"\cong"] \arrow[r,"\delta"] & H^{1}(BF_{n};I(F_{n})) \arrow[rd, "I(\phi)_{*}"]&  \\
& &  H^{1}(BF_{n};I(F_{m})) \arrow[dl, "s^{*}"] \arrow[bend right=60,swap]{ul}{I(s)_{*}}\\
H^{0}(BF_{m};\Z) \arrow [uu]\arrow[r, "\delta"] & H^{1}(BF_{m};I(F_{m})) \arrow[bend right=60,swap]{ru}{\phi^{*}} &    
\end{tikzcd}
$$

Note that we can treat $I(F_{n})$ is $F_{m}-$ module via $\phi$ while $I(F_{m})$ is $F_{n}-$ module via $s$. Hence, all homomorphisms are well-defined. The commutative diagram give us the conclusion of the corollary: $\delta(1)=\beta_{F_{n}}=I(s)_{*}\circ\phi^{*}(\beta_{F_{m}})$ and $\delta(1)=\beta_{F_{m}}=s^{*}\circ I(\phi)_{*}(\beta_{F_{n}})$ 
\end{proof}

Let $\Gamma=\pi_{1}(M_{g})$ and $\Lambda_{1}(M_{g'})$ be the fundamental groups of orientable, closed surfaces $M_{g}$ and $M_{g'}$ where $g,g'$ are positive genus numbers.

\begin{thm}\label{surface groups}
 Let $\phi:\Gamma\to\Lambda$ be an epimorphism of the fundamental groups of the closed surfaces of $M_{g}$ and $M_{g'}.$ Then $$\cat(\phi)=\cd(\phi)=1$$ if and only if $\phi$ factors through the free groups via epimorphisms $\Gamma\to F_{n}\to \Lambda.$ Otherwise, $\cat(\phi)=\cd(\phi)=2.$  
\end{thm}

\begin{proof}
By Proposition \ref{char. of group hom}, we always obtain $\cd(\phi)\geq 1$ since $\Gamma$ is not the trivial group. 

We consider two steps: 

{\em Step 1.} 
Now we suppose the epimorphism $\phi:\Gamma\to\Lambda$ factors through the free groups via epimorphisms $\Gamma\to F_{n}\to \Lambda.$ Since $\cd(\phi)\leq \cat(\phi)$ (see \cite{DK}), it suffices to show $\cat(\phi)\leq 1.$

This easily follows from the classical LS-category. Indeed, $$\cat(\phi)\leq \min\{\cat(h),\cat(g)\}\leq \cat(F_{n})=1,$$ where $h:B\Gamma\to BF_{n}$ and $g:BF_{n}\to B\Lambda$ are maps that induced epimorphisms $h_{*}:\Gamma\to F_{n}$ and $g_{*}:F_{n}\to \Lambda.$

Now conversely $\cat(\phi)=\cd(\phi)=1,$ then using the Ganea-Schwarz approach to LS-category of induced map $B\phi,$ we obtain that the induced map $B\phi:B\Gamma\to B\Lambda$ can be lifted to the Ganea's space $G_{1}(B\Lambda).$ Note that $G_{1}(B\Lambda)$ is homotopy equivalent to 1-dimensional complex $\Sigma\Lambda$, the reduced suspension of $\Lambda$ \cite{CLOT}. Then the epimorphism $\phi$ factors through a free group $F_{n}$ via epimorphisms $\Gamma\to F_{n}\to\Lambda.$    

{\em Step 2.} By the classical LS-category, $\cd(\phi)\leq \min\{\cat(\Gamma),\cat(\Lambda)\}=2$ since $\cat(\Gamma)=2$ and $\cat(\Lambda)=2.$ By {\em Step 1}, we have that $\cat(\phi)=\cd(\phi)\neq 1$ since the epimorphism $\phi$ does not factor through the free groups via epimorphisms $\Gamma\to F_{n}\to\Lambda$. Since $\cd(\phi)\geq 1$, we get the conclusion of Theorem $\cat(\phi)=\cd(\phi)=2.$  
\end{proof}

\begin{rmk}
We only use the surface groups in {\em Step 2.} Note that the proof of {\em Step 1} only requires a nonzero epimorphism and $\cat(\phi)=\cd(\phi)=1.$ Otherwise, consider any degree one map $f:T^{2}\to S^{2}$, then $\cat(f)=1$ and $\cd(f_{*})=0.$ Moreover, the epimorphism $f_{*}$ is not nonzero. In the case of aspherical spaces, there is also {\em the open problem} that states that there is a group epimorphism $\phi$ of finite aspherical spaces with $\cd(\phi)=1$ and $\cat(\phi)=2.$ (More details, see \cite{DK,DD}). That is why we need the condition $\cat(\phi)=\cd(\phi)=1$ for finite aspherical spaces instead of just $\cd(\phi)=1.$            
\end{rmk}

We need the product formula for cohomological dimension of group homomorphisms to prove the main results.

\begin{thm}[\cite{DD}]\label{DD} For every homomorphism $\phi:\Gamma\to\Lambda$ of geometrically finite groups $\cd(\phi\times\phi)=2\cd(\phi).$
\end{thm} 

\begin{cor}\label{product formula}
    For every homomorphism $\phi:\Gamma\to\Lambda$ of geometrically finite groups $\cd(\phi^{n})=n\cd(\phi),$ where $\phi^{n}=\phi\times\cdots\times\phi:\Gamma^{n}\to\Lambda^{n}.$
\end{cor}
\begin{proof}
    The first inequality $\cd(\phi^{n})\leq n\cd(\phi)$ follows by \cite[Lemma 6.1]{DD} and the induction on $n$.

    To show the inequality $\cd(\phi^{n})\geq n\cd(\phi).$
By induction on $k$ we obtain the equality $\cd(\phi^{2^{k}})=2^{k}\cd(\phi)$ from Theorem \ref{DD}.

For $n<2^{k}$ the equality 
$2^{k}\cd(\phi)=\cd(\phi^{2^{k}})\leq \cd(\phi^{n})+\cd(\phi^{2^{k}-n})\leq \cd(\phi^{n})+(2^{k}-n)\cd(\phi)$ implies $n\cd(\phi)\leq \cd(\phi^{n}).$
    
\end{proof}

\section{The $D-$topological complexity introduced by Farber and Oprea}

Let $\Updelta$ be the diagonal subgroup of $\Lambda^{r}.$ Let $D_{\Lambda}$ be the minimal family of subgroups of $\Lambda^{r}$ containing the diagonal subgroup $\Updelta$ and the trivial group, which is closed under conjugations and taking finite intersections.

In \cite{FO}, Farber and Oprea introduced $D-$topological complexity of a path-connected topological space. Here we modify their definition to give $D_{\Lambda}-$topological complexity of a map $f:X\to Y$ between path-connected topological spaces $X, Y$. 
Let $\pi_{1}(X,x_{0})=\Gamma$, $\pi_{1}(Y,y_{0})=\Lambda$ and $y_{0}=f(x_{0})$. 

\begin{defn}
The $D_{\Lambda}-$topological complexity, $TC_{r}^{D_{\Lambda}}(f),$ is defined as the minimal number $k$ such that $X^{r}$ is covered by k+1 open subsets $\{U_{0},\cdots,U_{k}\}$ with property that for each $U_{i}$ and for any choice of the base point $\bar{u}\in U_{i}$, the induced the homomorphism $f_{*}^{r}|_{U_{i}}:\pi_{1}(U_{i},\bar{u})\to \pi_{1}(Y^{r},f^{r}(\bar{u}))$ takes values in a subgroup conjugate to the diagonal $\Updelta$ of $\Lambda^{r}$.   

\end{defn}

\begin{rmk}\label{projection}
If the map is the identity map, then we recover Farber' and Oprea' invariant. This definition is also equivalent to saying that the projections from $f^{r}(Z)$ to the $j+1$ factor of $Y^{r}$ are homotopic, where $j=0,..,r-1.$   
\end{rmk}
Let $X$ and $Y$ be $CW-$complexes and let $X^{(n)}$ (respectively $Y^{(n)}$) denote the $n$-skeleton of $X$ (respectively $Y$).
We recall that a continuous map $f:X\to Y$ is said to be cellular if it takes n-skeletons of $X$ to $n$-skeletons of $Y$, $f(X^{(n)})\subset Y^{(n)}$, for all non-negative integers $n$. Note that we always assume that CW-complex has only one zero cell. 

We obtain the following the generalised version of Lemma 4.2 in \cite{FO} and Theorem 6.2 in \cite{EFMO}.
\begin{thm}\label{FO}
Let $f:X\to Y$ be a cellular map between finite aspherical spaces with $\pi_{1}(X)=\Gamma, \pi_{1}(Y)=\Lambda$. Then $TC_{r}^{D_{\Lambda}}(f)=TC_{r}(f)=\secat((f^{r})^{*}q)$ where $q:\hat{Y^{r}}\to Y$ be the connected covering space corresponding to diagonal subgroup $\Updelta_{\Lambda}$ of $\Lambda^{r}$.   
\end{thm}
\begin{proof}
Since $f:X\to Y$ is a cellular map, the proof of the first equality follows from Lemma 4.2 in \cite{FO} and Remark \ref{projection}.

The second equality follows from the following commutative diagram
$$\begin{tikzcd}
(f^{r})^{*}PY \arrow[r, "\bar{f^{r}}"] \arrow[d,"(f^{r})^{*}p_{r}"] & PY \arrow[r] \arrow[d,"p_{r}"] & \hat{Y^{r}} \arrow[l] \arrow[dl,"q"]\\
  X^{r} \arrow[r,"f^{r}"] & Y^{r} \\  
\end{tikzcd}
$$
We recall that the Deck transformation $\Lambda$ acts on the universal cover of Y, $\tilde{Y}$ \cite{Ha}. The product of Deck transformations, $\Lambda^{r}$ acts component-wise on the product of universal covers of $Y^{r}$, $\tilde{Y^{r}}.$ The orbit space $\hat{Y^{r}}:=\tilde{Y^{r}}/\Updelta_{\Lambda}$ where the action of $\Lambda^{r}$ on $\tilde{Y^{r}}$ restrict on the diagonal subgroup $\Updelta_{\Lambda}.$ 

Note that the fibrations $p_{r}$ and $q$ are a fiber homotopy equivalence since $\hat{Y^{r}}$ and $PY$ are aspherical spaces. Thus, the pull-back fibrations $(f^{r})^{*}p_{r}$ and $(f^{r})^{*}q$ are a fiber homotopy equivalent. Since $TC_{r}(f)=\secat((f^{r})^{*}p_{r})$ by Theorem \ref{sequential TC}, we get the second equality.  
    
\end{proof}

\begin{rmk}
By Theorem \ref{FO}, we always have the following inequalities $TC_{r}^{D_{\Lambda}}(f)\leq TC_{r}(f)$ for any map of cell complexes. This inequality can be strict in general. For example, take Hopf bundle $S^{1}\to S^{3}\overset{h}{\to} S^{2}.$ Since spaces are simply-connected, $TC_{r}^{D_{\Lambda}}(h)=0$. However, $TC_{r}(h)\geq 1$ since $h$ is not null-homotopic, $\cat(h)=1$ and $\cat(h)\leq \cat(h^{r-1})\leq TC_{r}(h)$ by Proposition \ref{TC}.  
\end{rmk}

\section{Group homomorphisms}

Due to the homotopy invariance of classifying spaces, there is a one-to-one correspondence between group homomorphism $\phi:\Gamma\to\Lambda$ and homotopy classes of maps $B\phi:B\Gamma\to B\Lambda$ that induce $\phi$ on the fundamental group. Hence, the following definition is well-defined. 

\begin{defn}
 Let $\phi:\Gamma\to\Lambda$ be a homomorphism.
 \begin{enumerate}
\item The LS-category of $\phi$, $\cat(\phi),$ is defined to be $\cat(B\phi).$
\item The sequential topological complexity of $\phi$, $TC_{r}(\phi),$ is defined to be $TC_{r}(B\phi).$
 \end{enumerate}
\end{defn}

\subsection{Reduction to epimorphisms}

Let $\phi:\Gamma\to\Lambda$ be a nonzero homomorphism. Let us denote $\pi:=\phi(\Gamma)$, $\psi:=\phi:\Gamma\to\pi$, and $i:\pi\to\Lambda$ is the inclusion homomorphism.  

With Dranishnikov \cite{DK}, we proved that $\cat(\phi)=\cat(\psi)$, $\cat(i)=\cat(\pi),$ $\cd(\phi)=\cd(\psi)$, and $\cd(i)=\cd(\pi).$ However, these phenomena do not hold for the sequential topological complexity of group homomorphisms in the general. Thus, we prove them under some assumption on groups.  

\begin{thm}\label{injective}
Let $\phi:\Gamma'\to \Gamma$ be an injective homomorphism and $\Gamma$ is an abelian group. Then $TC_{r}(\phi)=TC_{r}(\Gamma').$    
\end{thm}

\begin{proof}
By Proposition \ref{TC}, we get $TC_{r}(\phi)\leq TC_{r}(\Gamma').$ To prove other direction, we use the Ganea-Schwarz characterisation of the sequential topological complexity. Note that $TC_{r}(\phi)=\secat((\phi^{r})^{*}u_{\Updelta})$ and $TC_{r}(\Gamma')=\secat(u'_{\Updelta}).$ Existing lift $\tilde{\phi^{r}}$ of map $\phi^{r}$ with respect to $u_{\Updelta}$ gives a section $s$ of $u_{\Updelta},$ i.e. $\tilde{\phi^{r}}=s\phi^{r}.$ Because of pull-back diagram, this section $s$ is also the section of $(\phi^{r})^{*}u_{\Updelta}.$ Hence, we conclude that the composition of this section $s$ and the dash arrow is a section of $u'_{\Updelta}.$   

\begin{tikzcd}
B\Updelta_{\Gamma'} \arrow[dr,"u'_{\Updelta}"]& (\phi^{r})^{*}B\Updelta_{\Gamma} \arrow[r, "\bar{\phi^{r}}"] \arrow[l, shift left=0.5ex,dashed] \arrow[d,"(\phi^{r})^{*}u_{\Updelta}"]& B\Updelta_{\Gamma} \arrow[d, "u_{\Updelta}"]\\
&B(\Gamma')^{r} \arrow[r,"\phi^{r}"]& B\Gamma^{r}
\end{tikzcd}

Now it is left to prove the dash arrow exists. This is equivalent to the following lifting problem $(\phi^{r})^{*}u_{\Updelta}:(\phi^{r})^{*}B\Updelta_{\Gamma}\to B(\Gamma')^{r}$ for the given diagonal cover $u'_{\Updelta}:B(\Updelta_{\Gamma'})\to B(\Gamma')^{r}.$ Note that $B(\Updelta(\Gamma)')$ is homotopy equivalent to $B\Gamma'$. By lifting criterion in \cite{Ha}, the lift exists if and only if $(\phi^{r})^{*}(\Gamma)$ is subgroup of $\Updelta_{\Gamma'}$. Thus, the dashed arrow exists since the condition trivially satisfies when $\Gamma$ is the abelian group.

\end{proof}

\subsection{A group epimorphism}

We use the notation $T(A)$ for the torsion subgroup of an abelian group $A$.
For finitely generated abelian groups we define the rank $rank(A)=rank(A/T(A))$.

We recall the Invariant Factor Decomposition for finite abelian group $T(A)$. 
\begin{thm}[Invariant Factor Decomposition (IDF) for Finite Abelian Groups]{\label{IDF}} Every finite abelian group $T(A)$ can be written uniquely as $T(A)=Z_{n_{1}}\times...\times Z_{n_{k}}$ where the integers $n_{i}\geq 2$ are the invariant factors of $T(A)$ that satisfy $n_{1}|n_{2}|...|n_{k}$ and $\Z_{n_{i}}$ are cyclic group of order $n_{i}, i=1,\cdots,k.$
\end{thm}

\begin{defn}
Given a finite abelian group $T(A)$, the Smith Normal number $k(T(A))$ of $T(A)$ is the number $k$ from Theorem ~\ref{IDF}.     
\end{defn}

\begin{thm}\label{FGAG}
Let $\phi:\Gamma\to \Lambda$ be an epimorphism of finitely generated abelian groups. 
Then $$TC_{r}(\phi)=(r-1)\cd(\phi).$$ In particular, if $\Gamma$ is free abelain groups, then $TC_{r}(\phi)=(r-1)\cd(\phi)=(r-1)(\rank(\Lambda)+k(T(\Lambda)))$ where $k(T(\lambda))$ is the Smith Normal number for given finite abelian group $T(\Lambda).$  
\end{thm}

\begin{proof}
Since the groups $\Gamma$ and $\Lambda$ are finitely generated abelian groups, we may assume $\Gamma=\Z^{n}\oplus T(\Gamma)$ and $\Lambda=\Z^{m}\oplus T(\Lambda)$ for some $m$ and $n$. 

By Theorem 2 in \cite{Ku1}, if both groups $\Gamma, \Lambda$ have torsion and $\phi(T(\Gamma))\neq 0$, then $\cd(\phi)=\infty.$  

Since $\cd(\phi)\leq\cat(\phi)$ and $\cat(\phi)\leq TC_{r}(\phi),$ we get $TC_{r}(\phi)=\infty$ and equality holds trivially. 

Thus, we consider $\phi$ with $\phi(T(\Lambda))=0$.
Such $\phi$ factors through the epimorphism $\bar\phi:\Z^n\to\Lambda$. 
In view of the retraction $B\Gamma\to B\Z^n$, we obtain that
$\cd(\bar\phi)=\cd(\phi)$ and $TC_{r}(\bar\phi)=TC_{r}(\phi)$ by Corollary \ref{retraction}. Therefore, we may assume that $\Gamma=\Z^n$.

Since $B\Z^{n}$ is an $n$ dimensional torus (Topological group), we apply Theorem \ref{H space} . Hence, we get that $TC_{r}(\phi)=\cat(\phi^{r-1}).$ By Theorem 1 in \cite{Ku1}, the analogue of Eilenberg and Ganea Theorem holds for finitely generated abelian groups, i.e. $\cat(\phi^{r-1})=\cd(\phi^{r-1}).$ 

By Corollary \ref{product formula}, $\cd(\phi^{r-1})=(r-1)\cd(\phi)$ and by Theorem 2 in \cite{Ku1},  $\cd(\phi)=\rank(\Lambda)+k(T(\Lambda))$

Combining these all equalities, we get the theorem:

$TC_{r}(\phi)=\cat(\phi^{r-1})=\cd(\phi^{r-1})=(r-1)\cd(\phi)=(r-1)(rank(\Lambda)+k(T(\Lambda))).$
\end{proof}

\begin{rmk}
  Note that even $r=2,$ this theorem is the generalized result of \cite[Theorem 6.12 and Colollary 6.13]{Sc}.    
\end{rmk}

\begin{thm}\label{Free}
Let $\phi:\Gamma\to\Lambda$ be a nonzero epimorphism of free groups. Then $$TC_{r}(\phi)=\begin{cases} r-1 & \mbox{if} \,\,\, \Lambda\cong\Z \\ r & \mbox{if} \,\,\,  \Lambda\cong F_{m}, m>1\end{cases}$$
\end{thm}

\begin{proof} We prove in two steps:

{\em Step 1.} Suppose $\Lambda\cong \Z.$

The upper bound: By Proposition \ref{TC}, $TC_{r}(\phi)\leq TC_{r}(S^{1})=r-1$ since $S^{1}=K(\Z,1)$ and $TC_{r}(S^{1})=r-1$ by \cite{Rud}.

The lower bound: By Proposition \ref{TC} $\cat(\phi^{r-1})\leq TC_{r}(\phi)$. Since $\cd(\phi^{r-1})=(r-1)\cd(\phi)$ by Corollary \ref{product formula} and $\cat(\phi)=\cd(\phi)=1$ by Corollary \ref{free group}, we get that $r-1\leq \cat(\phi^{r-1}).$ Hence, we obtain that $TC_{r}(\phi)=r-1.$
    
{\em Step 2.} $\Lambda\cong F_{m}$ where $F_{m}$ is the free group with $m>1$ generators. Let $\Gamma\cong F_{n}$ where $n\geq m$ since $\phi$ is surjective.

The upper bound: Similar to {\em Step 1}, $TC_{r}(\phi)\leq r$ since $TC_{r}(F_{n})=r$ by \cite{FO}. 

The lower bound: We apply the Theorem \ref{cup length} to show that $TC_{r}(\phi)\geq r.$

We first recall that the canonical cohomological class $$\alpha_{(F_{m})}\in H^{1}(BF_{m}^{r};I(F_{m}^{r-1}))$$ is given by a crossed homomorphism $f_{r,m}:F_{m}^{r}\to I(F_{m}^{r-1})$ such that $f_{r,m}(g_{1},g_{2},\cdots,g_{r})=(g_{1}g_{2}^{-1}-1,g_{2}g_{3}^{-1}-1,,\cdots,g_{r-1}g_{r}^{-1}-1)$ in \cite{EFMO}. Note that the crossed homomorphism $f_{r,m}$ restricted the diagonal $\Updelta_{F_{m}}$ of $F_{m}^{r}$ is trivial, so $\alpha_{(F_{m})}\in \ker(\Updelta_{F_{m}}^{*}).$

Similarly to Corollary \ref{free group}, the canonical cohomological classes of $\alpha_{(F_{m})}\in H^{1}(BF_{m}^{r};I(F_{m}^{r-1}))$ and $\alpha_{(F_{n})}\in H^{1}(BF_{n}^{r};I(F_{n}^{r-1}))$ are related. Indeed, we have the following commutative diagrams:

$$\begin{tikzcd}
 F_{n}^{r} \arrow[rr, "f_{r,n}"] \arrow[d, "\phi^{r}"] & & I(F_{n}^{r-1}) \arrow[bend left=60]{d}{I(\phi^{r-1})}\\
 F_{m}^{r} \arrow[rr, "f_{r,m}"] \arrow[bend left=60]{u}{s^{r}} & & I(F_{n}^{r-1}) \arrow[bend left=60]{u}{I(s^{r-1})}
\end{tikzcd}
$$

where $s:F_{m}\to F_{n}$ is a section for the epimorphism $\phi:F_{n}\to F_{m},$ $I(s^{r-1}), I(\phi^{r-1}),$ $s^{r}$, and $\phi^{r}$ are the obvious morphisms.

Since $\alpha_{(F_{m})}\neq 0$, we obtain that $(\phi^{r})^{*}:H^{1}(BF_{m}^{r};I(F_{m}^{r-1}))\to H^{1}(BF_{n}^{r};I(F_{m}^{r-1}))$ is not trivial homomorphism,i.e. $(\phi^{r})^{*}\alpha_{(F_{m})}\neq 0.$ 

Note that $(\alpha_{(F_{m})})^{r}\neq 0$ in $H^{r}(BF_{m}^{r};I(F_{m}^{r-1})^{\otimes r})$ where $$I(F_{m}^{r-1})^{\otimes r}=I(F_{m}^{r-1})\otimes I(F_{m}^{r-1})\otimes \cdots \otimes I(F_{m}^{r-1}).$$

Thus, $$(\phi^{r})^{*}(\alpha_{(F_{m})})^{r}\neq 0.$$ By Theorem \ref{cup length}, we obtain that $r\leq TC_{r}(\phi).$ Hence, we conclude that $TC_{r}(\phi)=r$ when $\Lambda\cong F_{m}$ and $m>1.$

\end{proof}

We extend the results for other discrete groups using author's previous work on LS-category of group homomorphisms \cite{DK,Ku1,Ku2} and Corollary \ref{product formula}, so we give the sharp estimates. This phenomena is expected since the exact value of the sequential topological complexity of discrete group $\Gamma,$ $TC_{r}(\Gamma)$ is known for a few classes of groups (See \cite{FO,EFMO,BGRT,Rud}).   

\begin{thm}
 Let $\phi:\Gamma\to\Lambda$ be an epimorphism of surface groups. Then $$r-1\leq TC_{r}(\phi)\leq r \,\,\,\,\textit{if} \,\,\,\, \cat(\phi)=\cd(\phi)=1.$$ Otherwise, $2(r-1)\leq TC_{r}(\phi)\leq 2r$ if $\cat(\phi)=\cd(\phi)=2.$
\end{thm}

\begin{proof}
 By Proposition \ref{TC}, we obtain that $\cat(\phi^{r-1})\leq TC_{r}(\phi)$ Since $\cd(\phi^{r-1})\leq \cat(\phi^{r-1})$ and by Corollary \ref{product formula}, we obtain that $\cd(\phi^{r-1})=(r-1)\cd(\phi)=r-1$ where $\cd(\phi)=1.$  
 
 Since $\cat(\phi)=\cd(\phi)=1,$ by Theorem \ref{surface groups}  the epimorphism $\phi$ factors through the free groups $F_{n}$ via epimorphisms $\Gamma\to F_{n}\to \Lambda.$
By Proposition \ref{Composition}, we obtain that $$TC_{r}(\phi)\leq \min\{TC_{r}(h),TC_{r}(g)\}\leq TC_{r}(F_{n})=r$$ where the last equality is due {\cite{FO}}, $h:\Gamma\to F_{n}$ and $g:F_{n}\to\Lambda.$

With similar above argument, we obtain that $2(r-1)\leq TC_{r}(\phi)\leq 2r$ if $\cat(\phi)=\cd(\phi)=2.$
\end{proof}

We recall the upper central series of a group $\Gamma$ is a chain of subgroups 
$${e}=Z_{0} \leqslant Z_{1} \leqslant....\leqslant Z_{n} \leqslant.... $$
where $Z_{1}=Z(\Gamma)$ is the center of the group, and $Z_{i+1}$ is the preimage under the canonical epimorphism $\Gamma\to\Gamma/Z_{i}$ of the center of $\Gamma/Z_{i}$.
A group $\Gamma$ is {\em nilpotent} if $Z_{n}=\Gamma$ for some $n$.

A group $\Gamma$ is called a \textit{virtually nilpotent} if it has a nilpotent subgroup $\Gamma'$ of finite index. 

\begin{thm}\label{nilpotent}
 Let $\phi:\Gamma\to\Lambda$ be an epimorphisms between finitely generated, torsion-free nilpotent (virtually nilpotent) groups. Then $$(r-1)\cd(\Lambda)\leq TC_{r}(\phi)\leq r\cd(\Lambda).$$   
\end{thm}
\begin{proof}
By Proposition \ref{TC}, we obtain that $\cat(\phi^{r-1})\leq TC_{r}(\phi)\leq \cat(\phi^{r}).$ Since the Cartesian product of nilpotent (virtually nilpotent) groups is nilpotent (virtually nilpotent), Theorem 4.1 in \cite{DK} gives the following equality, $\cat(\phi^{k})=\cd(\phi^{k})$ for each finite k. By Corollary \ref{product formula}, $\cd(\phi^{k})=k\cd(\phi)$ since the groups are geometrically finite. Since $\cd(\phi)=\cd(\Lambda)$ in Theorem 4.2 \cite{Ku2}, we obtain the result by the combining all above equalities.    
    
\end{proof}

We recall an almost nilpotent group is the extension group of the infinite cyclic group by the torsion-free nilpotent groups. More explicitly, one can think that it is the semidirect product of the torsion-free nilpotent group with infinite cyclic group $\Z\ltimes_{\phi}\Gamma,$ where $\phi:\Z\to Aut(\Gamma)$ is a one-parameter group.

\begin{thm}
Let $f:\Z\ltimes_{\phi}\Gamma\to\Z\ltimes_{\nu}\Lambda$ be an epimorphism between torsion-free almost nilpotent groups such that $f|_{\Z}=Id.$ Then $$(r-1)(\cd(\Lambda)+1)\leq TC_{r}(\phi)\leq r(\cd(\Lambda)+1).$$     
\end{thm}
\begin{proof}
Similarly to Theorem \ref{nilpotent}, we have that $\cat(\phi^{r-1})\leq TC_{r}(\phi)\leq \cat(\phi^{r}).$ 
Since $\cat(f)=\cd(f)=\cd(\Lambda)+1$ in \cite{Ku2} and Corollary \ref{product formula}, it suffices to show $\cat(f^{k})=\cd(f^{k}).$

Note that we always have the following inequality $\cd(f^{k})\leq\cat(f^{k})$ (see \cite{DK}).

The reverse inequality comes from the product formula for LS-category of maps and Corollary \ref{product formula}. Indeed, $\cat(f^{k})\leq k\cat(f)=k\cd(f)=\cd(f^{k})$.

\end{proof}

We would like to end the paper by asking the following questions:
\begin{ques}\label{cohomological bound}
 Is the following equality $\cd(\phi^{r-1})\leq TC_{r}(\phi)\leq \cd(\phi^{r})$ true for the epimorphisms of geometrically finite groups?   
\end{ques}
This question is motivated by the following fact $\cd(\Gamma^{r-1})\leq TC_{r}(\Gamma)\leq\cd(\Gamma^{r}).$ 

\begin{ques}\label{gap}
  What is the gap of LS-category and cohomological dimension of group homomorphisms of geometrically finite groups?   
\end{ques}

 All known examples regarding the Question \ref{gap} have a gap of 1 (See \cite{DK,DD,Gr}). Possible negative answer for Question \ref{cohomological bound} comes the constructing the arbitrary large gap for Question \ref{gap}.

\section*{Acknowledgments}
I would like to thank my advisor, Alexander Dranishnikov, for all of his help and encouragement
throughout this project.


\footnotesize
\begin{thebibliography}{999999}
\bibliographystyle{alpha}

 
\bibitem[Be]{Be}
Berstein–Schwarz,  On the Lusternik-Schnirelmann category of Grassmannians. Math. Proc.
Camb. Philos. Soc. 79  (1976) 129-134.

\bibitem[BGRT]{BGRT} I. Basabe, J. Gonz´alez, Y. B. Rudyak, D. Tamaki, Higher topological complexity and
its symmetrization, Algebr. Geom. Topol., 14 (2014), 2103-2124.


\bibitem[Br]{Br} K. Brown, Cohomology of Groups. \emph{Graduate Texts in Mathematics},
\textbf{87} Springer, New York Heidelberg Berlin, 1994.


\bibitem[CLOT]{CLOT}
    O. Cornea, G. Lupton, J. Oprea, D. Tanre,
\newblock   { Lusternik-Schnirelmann Category},  AMS,  2003.

\bibitem[DD]{DD} Aditya De Saha, Alexander Dranishnikov, On cohomological dimension of group homomorphisms
preprint, arXiv:2302.09686 [math.AT] (2023)

\bibitem[DJ]{DJ} A. Dranishnikov, E. Jauhari, Distributional topological complexity and LS-category,
 arXiv:2401.04272 [math.GT] (2024), 21 pp.

\bibitem[DK]{DK} A. Dranishnikov, N. Kuanyshov, On the LS category of group homomorphisms, Math. Z. 305, no. 1 (2023): 14.


\bibitem[DR]{DR} A. Dranishnikov, Yu. Rudyak, On the Berstein–Schwarz-Svarc theorem in dimension 2. Math. Proc. Cambridge Philos. Soc. 146 (2009), no. 2, 407-413.

\bibitem[DS]{DS} A. Dranishnikov, R. Sadykov, The Lusternik–Schnirelmann category of a connected sum, Fundamenta Mathematicae 251 (2020), no. 3, 313-328. 

\bibitem[EFMO]{EFMO} Arturo Espinosa, Michael Farber, Stephan Mescher, and John Oprea. "Sequential topological complexity of aspherical spaces and sectional categories of subgroup inclusions." arXiv preprint arXiv:2312.01124 (2023).

\bibitem[EG]{EG} S. Eilenberg, T. Ganea, {\em On the Lusternik-Schnirelmann Category of Abstract Groups.} Annals of Mathematics, 65, (1957), 517-518.

\bibitem[Fa1]{Fa1} M. Farber, Topological complexity of motion planning, Discrete Comput. Geom. 29 (2003) 211-221.

\bibitem[Fa2]{Fa2} M. Farber, Topology of robot motion planning, in: Morse Theoretic Methods in Nonlinear Analysis and in Symplectic Topology, 2006, pp. 185-230. 

\bibitem[Fa3]{Fa3}  M. Farber, Instabilities of robot motion, Topology and its Applications 140 (2004) 245–266.

\bibitem[FO]{FO} M. Farber, J. Oprea {\em Higher topological complexity of aspherical spaces}
Topol. Appl., 258 (2019), pp. 142-160

\bibitem[Gr]{Gr}
    M. Grant, 
\newblock    { https://mathoverflow.net/questions/89178/cohomological-dimension-o$f$-a-homomorphism}

\bibitem[Ha]{Ha} A. Hatcher, Algebraic topology, 2005.

\bibitem[Ja]{Ja}  I. M. James, On category in the sense of Lusternik-Schnirelmann, Topology, 17 (1978), 331-348.

\bibitem[Ju]{Ju} E. Jauhari. ``On Sequential Versions of Distributional Topological Complexity.” preprint,
arXiv:2401.15667 [math.AT] (2024), 27 pp.

\bibitem[Ku1]{Ku1} N. Kuanyshov, On the LS-category of homomorphisms of groups with torsion, Algebra and Discrete Mathematics, Volume 36 (2023). Number 2, pp. 166–178
DOI:10.12958/adm2065

\bibitem[Ku2]{Ku2} N. Kuanyshov, On the LS-category of homomorphism of almost nilpotent groups, Topology and its Applications, Volume 342, 1 February 2024, 108776.

\bibitem[KL]{KL} L. Kavraki, S. LaValle. Motion planning. Chapter 5 of Handbook of Robotics. Springer-Verlag, Berlin Heidelberg 2008.

\bibitem[KW]{KW}  B. Knudsen, S. Weinberger, Analog category and complexity, preprint,
arXiv:2401.15667 [math.AT] (2024), 19 pp.


\bibitem[LS]{LS}  L. Lusternik, L. Schnirelmann, ``Sur le probleme de trois geodesiques fermees sur les surfaces de genre 0", Comptes Rendus de l'Academie des Sciences de Paris, 189: (1929) 269-271.

\bibitem[La]{La}] J. Latombe, Robot Motion Planning, Kluwer Academic, Dordrecht, 1991.

\bibitem[MW]{MW} Murillo A, Wu J. Topological complexity of the work map. Journal of Topology and Analysis 2021; 13 (01): 219-238.

\bibitem[Pa1]{Pa1} P. Pavesic, Topological complexity of a map, Homol. Homotopy Appl. 21 (2019) 107-130.

\bibitem[Pa2]{Pa2} P. Pavesic, A topologist's view of kinematic maps and manipulation complexity, Contemp. Math. 702 (2018) 61-83. 

\bibitem[Rud]{Rud} Yu. Rudyak On higher analogs of topological complexity. Topology and its Applications 2010; 157 (5): 916-920. Erratum in Topology and its Applications 2010; 157 (6): 1118.

\bibitem[RS]{RS} Yu. Rudyak, S. Soumen. Relative LS categories and higher topological complexities of maps. Topology and its Applications 2022; 322: 108317.

\bibitem[Sch]{Sch}  A. Schwarz, The genus of a fibered space. Trudy Moscov. Mat. Obsc. 10, 11 (1961 and 1962), 217-272, 99-126.

\bibitem[Sc]{Sc} J. Scott. {\em On the topological complexity of maps}, Topology and its Applications 314 (2022), Paper No. 108094, 25 pp.

 \bibitem[Sp]{Sp} E. Spanier, Algebraic topology, Springer Science \& Business Media 1989.
  

 \bibitem[ZG]{ZG} C. ZAPATA, J. GONZÁLEZ.``Higher topological complexity of a map," Turkish Journal of Mathematics: (2023) Vol. 47: No. 6, Article 3.  https://doi.org/10.55730/1300-0098.3453; 
\end{thebibliography}
\end{document}